\documentclass[11pt]{amsart}

\author{Clifford Gilmore}
\address{Department of Mathematics and Statistics\\ P.O. Box 68\\ FI-00014 University of Helsinki\\ Finland.}
\email{clifford.gilmore@helsinki.fi}

\author{Eero Saksman}
\address{Department of Mathematics and Statistics\\ P.O. Box 68\\ FI-00014 University of Helsinki\\ Finland.}
\email{eero.saksman@helsinki.fi}

\author{Hans-Olav Tylli}
\address{Department of Mathematics and Statistics\\ P.O. Box 68\\ FI-00014 University of Helsinki\\ Finland.}
\email{hans-olav.tylli@helsinki.fi}

\thanks{The first and second authors have been supported by the Academy of Finland via the Centre of Excellence in Analysis and Dynamics Research (project no. 271983).  The first author has also been supported by the Magnus Ehrnrooth Foundation.}

\title{Hypercyclicity properties of commutator maps}
\date{\today}

\usepackage[UKenglish]{babel}

\usepackage{amssymb,mathrsfs,mathtools}

\usepackage[all]{xy} 	

\usepackage{enumerate}

\newtheorem{thm}{Theorem}[section]

\newtheorem{cor}[thm]{Corollary}
\newtheorem{prop}[thm]{Proposition}

\theoremstyle{definition}
\newtheorem{eg}[thm]{Example} 
\newtheorem{rmk}[thm]{Remark}

\numberwithin{equation}{section}		

\newcommand{\tr}[1]{\mathrm{tr}\!\left( #1 \right)}

\newcommand{\cspn}[1]{\overline{\mathrm{span}}\!\left\lbrace #1 \right\rbrace} 

\DeclareMathOperator{\ran}{ran}

\begin{document}

\begin{abstract}
We investigate the hypercyclic properties of commutator maps acting on separable ideals of operators.  
As the main result we prove the commutator map induced by scalar multiples of the backward shift operator fails to be hypercyclic on the space of compact operators on $\ell^2$.
We also establish some necessary conditions  which identify large classes of operators that do not induce hypercyclic commutator maps.
\end{abstract}

\maketitle

\section{Introduction}

Let $X$ be a Banach space and $\mathscr{L}(X)$ the space of bounded linear 
operators on $X$. The \emph{commutator operator} 
$\Delta_T \colon \mathscr{L}(X) \to \mathscr{L}(X)$ induced by a fixed bounded linear operator
$T \in \mathscr{L}(X)$ is defined as
\begin{equation*}
S \mapsto \Delta_T(S) =  TS - ST = L_T(S) - R_T(S)
\end{equation*}
where $S \in \mathscr{L}(X)$ and $L_T, R_T \colon \mathscr{L}(X) \to \mathscr{L}(X)$ are, respectively, the left and right multiplication operators.

Recall for a separable Banach space $X$ that the operator $U \in \mathscr{L}(X)$ is 
\emph{hypercyclic} if there exists a vector $x \in X$ (said to be a hypercyclic vector for $U$)
such that its orbit under $U$ is dense in $X$, that is
\begin{equation*}
\overline{\{U^n x : n \geq 0\}} = X.
\end{equation*}
The purpose of this paper is to initiate the investigation of hypercyclicity properties of commutator maps $\Delta_T$ restricted to separable Banach ideals of operators,
which  turns out to be quite a subtle question.

The motivation for this study is at least twofold. Firstly,  Bonet et al.~\cite{BMP04} characterised the hypercyclicity of the left and right multiplication operators 
on such Banach ideals. Subsequently Bonilla and Grosse-Erdmann~\cite{BGE12} identified a sufficient condition for the frequent hypercyclicity of the left multiplier.  
This raises the question of the hypercyclicity properties of more complicated operators built up from the basic multipliers $L_T$ and $R_T$.
The next natural class of operators to consider are the commutator maps
$\Delta_T =  L_T - R_T$,  which have been extensively studied from various perspectives.  

Secondly, it is known that  on non-separable spaces such as $\mathscr{L}(\ell^2)$ 
the closure of the range $\overline{\ran} \left(\Delta_T \right)$ of the commutator maps  $\Delta_{T}$
is quite small.  
For instance, Stampfli~\cite[p. 519]{Sta73} showed that the quotient 
\begin{equation} \label{stampfli}
\mathscr{L}(\ell^2)/\overline{\ran} \left(\Delta_T \right)
\end{equation}
is not even separable for
any $T \in \mathscr{L}(\ell^2)$ and this fact was later extended to $X = \ell^p$, for $1 < p < \infty$,
in~\cite{ST98}.
Furthermore, it is well known that the identity map
$I_X \notin \ran \left(\Delta_T \right)$ for any  $T \in \mathscr{L}(X)$ when $X$ is infinite-dimensional.

On the other hand, by a celebrated example of Anderson~\cite{And73a} there also exist operators  $T \in \mathscr{L}(\ell^2)$ such that  $I_{\ell^2} \in \overline{\ran} \left(\Delta_T \right)$.
Furthermore, restrictions of commutator maps to separable ideals behave quite 
differently from  \eqref{stampfli}, for instance for the backward shift $B \in \mathscr{L}(\ell^2)$  
the restricted map $\Delta_B \colon \mathscr{K}(\ell^2) \to \mathscr{K}(\ell^2)$ has dense range,
where $\mathscr{K}(\ell^2)$ is the ideal of the compact operators 
on $\ell^2$. 
This provides further motivation to pursue this study.

In Sections \ref{sec:NonHcOpers} and \ref{sec:NonHcOpersHilbert} we first isolate 
large classes  of operators that do not induce hypercyclic commutator maps on separable ideals 
of $\mathscr{L}(X)$.  
The backward shift $B$ on $\ell^2$ is an obvious candidate for inducing a hypercyclic commutator map, 
but in Section \ref{sec:commBwNonHc}
we prove as our main result that $\Delta_{cB}$ is not hypercyclic 
on $\mathscr{K}(\ell^2)$ for any constant $c$. In fact, the same is also true for 
$\Delta_{p(B)}$, where $p$ is  any analytic polynomial.

\section{Background and Setting}  \label{sec:BgMultipliers}

Since the space of bounded linear operators  is non-separable under the operator norm topology for classical Banach spaces $X$, our setting will be that of separable ideals of $\mathscr{L}(X)$.
It is convenient to say that
$\left(J, \lVert \: \cdot \:  \rVert_J \right)$ is a \emph{Banach ideal} of $\mathscr{L}(X)$ if
\begin{enumerate}[(i)]
\item $J \subset \mathscr{L}(X)$ is a linear subspace,

\item the norm $\lVert \: \cdot \: \rVert_J$ is complete in $J$ and $\lVert S \rVert \leq  \lVert S \rVert_J$ for all $S \in J$, 

\item $BSA \in J$ and $\lVert BSA \rVert_J \leq   \lVert B \rVert  \lVert A \rVert  \lVert S \rVert_J$, for $A, B \in \mathscr{L}(X)$ and $S \in J$, 

\item the rank one operators $x^* \otimes x \in J$ and $\lVert x^* \otimes x \rVert_J =   \lVert x^* \rVert   \lVert  x \rVert$ for all $x^* \in X^*$ and $x \in X$. 
\end{enumerate}

Classical examples of Banach ideals are the space of nuclear operators $\left( \mathscr{N}(X), \lVert \: \cdot \: \rVert_N \right)$ with the nuclear norm and the spaces of 
approximable operators $\mathscr{A}(X)$ as well as compact operators $\mathscr{K}(X)$ 
under the operator norm.  
When $X$ is a Hilbert space the spaces of Schatten $p$-class operators $\left( C_p, \lVert \: \cdot \: \rVert_p \right)$ with the Schatten norm, for $1 \le p < \infty$,
are important instances of Banach ideals.

In the setting of Banach ideals, Bonet et al.~\cite{BMP04} characterised the hypercyclicity of the left and right multipliers using tensor techniques developed in \cite{MP03}. 
For a separable Banach ideal $J \subset \mathscr{L}(X)$ containing the finite rank operators as a dense subset  they showed that  

\begin{enumerate}[1.]
\item $L_T$ is hypercyclic on $J$ if and only if $T \in \mathscr{L}(X)$ satisfies the Hypercyclicity Criterion,

\item $R_T$ is hypercyclic on $J$ if and only if $T^*$ satisfies the Hypercyclicity Criterion on the dual  $X^*$.
\end{enumerate}

Instances  of  Banach ideals $J$ to which the above results apply include $\left( \mathscr{N}(X), \lVert \: \cdot \: \rVert_N \right)$ and $\mathscr{A}(X)$ 
when $X^*$ is separable, as well as  
$\mathscr{K}(X)$ when $X$ possesses the approximation property and $X^*$ is separable.  
Further details on these spaces and  the approximation property can be found for instance
 in \cite{Pie80} or \cite{Rya02}.
Although not explicitly stated in~\cite{BMP04}, we note that their results also 
yield sufficient conditions for the hypercyclicity of the two-sided multipliers  $L_T R_U$, for $T, U \in \mathscr{L}(X)$. (To see this  identify $L_T R_U$ with its tensor representation $U^* \otimes T$, whence  hypercyclicity follows directly from the sufficient conditions identified in \cite{MP03}.)

The area of linear dynamics contains an extensive body of work and we refer to ~\cite{BM09} and~\cite{GEP11} for expositions of the fundamental  results related to hypercyclicity.
We briefly recall here that the aforementioned Hypercyclicity Criterion is a sufficient condition for hypercyclicity
 which was initially demonstrated by Kitai~\cite{Kit82} and later independently rediscovered 
 by Gethner and Shapiro~\cite{GS87} .
We say that $T \in \mathscr{L}(X)$ satisfies the Hypercyclicity Criterion if there exist dense subsets $X_0$, $Y_0 \subset X$, an increasing sequence $(n_k)$ of positive integers and maps $S_{n_k}\colon Y_0 \to X$, $k \geq 1$, such that for any $x \in X_0$, $y \in Y_0$ one has
\begin{enumerate}[(i)]
\item $T^{n_k} x \to 0$,
\item $S_{n_k}(y) \to 0$, 
\item $T^{n_k}S_{n_k}(y) \to y$ 
\end{enumerate}
as $k \to \infty$.

We note in the literature there are also analogous hypercyclicity results on spaces of operators endowed with weaker topologies. For instance, hypercyclicity of the left  multiplier $L_T$ acting on
$\mathscr{L}(X)$ endowed with the strong operator topology  was 
characterised by Chan and Taylor~\cite{Cha99,CT01}.  Subsequently this was extended to supercyclicity by Montes-Rodr{\'i}guez and Romero-Moreno~\cite{MRRM02}.
Petersson~\cite{Pet07}  gave sufficient conditions for the hypercyclicity of the specific two-sided multipliers $L_{T^*}R_{T}$  on the space of self-adjoint operators 
on a Hilbert space under the  topology of uniform convergence on compact sets. 
Later Gupta and Mundayadan~\cite{GM15}  extended~\cite{Pet07} to the supercyclic case.
It is clearly natural to raise questions about hypercyclic properties of commutator maps 
$\Delta_T$ on $\mathscr{L}(X)$ with respect to weaker topologies however this alternative will not be explored in this paper.

\section{Classes of Non-Hypercyclic Commutator Maps}  \label{sec:NonHcOpers}

In this section we record three general observations which yield that 
commutator maps $\Delta_T$ are not hypercyclic for large classes of operators.
The first two observations are related to the well known fact that the adjoint 
of a hypercyclic operator has no eigenvalues (cf.~\cite[Proposition 1.17]{BM09}). 

Recall that for any Banach space $X$, the ideal of  nuclear operators $\mathscr{N}(X)$  is the smallest Banach ideal of
$\mathscr{L}(X)$  (cf. \cite[Theorem 6.7.2]{Pie80}).
Our first result implies that $\mathscr{N}(X)$ does not support any hypercyclic commutator
maps  under general conditions on $X$. 
For this recall  if $X^*$ has the approximation property then
$\mathscr{N}(X)^* = \mathscr{L}(X^*)$ in the trace duality 
\begin{equation}\label{dual}
\langle S, U \rangle = \tr{S^*U}, \quad U \in \mathscr{L}(X^*), \; S \in \mathscr{N}(X)
\end{equation}
(cf.  \cite[Theorem 1.e.4, Theorem 1.e.7]{LT77} or \cite[Theorem 10.3.2]{Pie80}).  
Here $\tr{S^*U} = \sum_{j=1}^\infty Ux_j^*(x_j)$ is the trace of $S^*U$,
for $S = \sum_{j=1}^\infty x_j^* \otimes x_j \in \mathscr{N}(X)$ and 
$U \in \mathscr{L}(X^*)$. This trace duality  is of course well known for $X = \ell^2$. 

Our first observation actually concerns the non-cyclicity of certain commutator maps. 
Recall for a separable Banach space $X$ that  $U \in \mathscr{L}(X)$ is a \emph{cyclic} operator if there exists $x \in X$ such that the linear span of its orbit under $U$ is dense in X, that is
\begin{equation}  \label{cyclic}
\cspn{U^n x : n \geq 0} = X.
\end{equation}
Clearly any hypercyclic operator is cyclic,
 but the cyclic operators form a much larger class.  For instance the forward shift 
 $S \in \mathscr{L}\left( \ell^2(\mathbb{N}) \right)$ is cyclic but  not hypercyclic. 

\begin{prop}  \label{thm:noHcNucs}
Let $X$ be any Banach space such that $X^*$ has the approximation property. Then the space of nuclear operators $\mathscr{N}(X)$ does not support any cyclic (nor any hypercyclic) 
commutator maps.
\end{prop}

\begin{proof}
We begin by recalling a simple necessary condition for cyclicity.
Suppose that $U \in \mathscr{L}(X)$ satisfies $\dim \left( \ker(U^*) \right) \geq 2$, that is,
there are linearly independent functionals $\{x_1^*, x_2^*\} \subset X^*$
for which ${x_j^*}_{| \ran(U)} = 0$ for $j = 1, 2$. In particular, this means that 
the codimension of the closure 
$\overline{\ran} \left(U \right)$ in $X$ is at least $2$. Then it follows that  $U$ cannot be a cyclic operator 
on $X$, since  \eqref{cyclic} implies that the codimension of $\overline{\ran} \left(U \right)$ 
in $X$ is at most $1$.

Let $T \in \mathscr{L}(X)$ be arbitrary.  Recall next that the adjoint of $\Delta_T$ is
\begin{equation*}
\left(\Delta_T \colon \mathscr{N}(X) \to \mathscr{N}(X) \right)^* = 
\Delta_{T^*}  \colon \mathscr{L}(X^*) \to \mathscr{L}(X^*)
\end{equation*}
in the trace duality of  \eqref{dual}. 
In fact, for any $x \in X$ and $x^* \in X^*$ we have that 
\begin{align*}
\langle x^* \otimes x,  & \, \Delta_T^*(U) \rangle = \langle x^* \otimes Tx -  T^*x^* \otimes x, \, U \rangle \\
&= \tr{U^*Tx \otimes  x^* - U^*x \otimes T^*x^*}  = \langle Ux^*, \, Tx\rangle - \langle T^*x^*, \, U^*x \rangle \\
& = \tr{(x \otimes x^*)(T^*U-UT^*)} = \langle x^* \otimes x, \, \Delta_{T^*}(U)\rangle.
\end{align*}
If $T = \lambda I_X$ for some $\lambda \in \mathbb C$, then obviously
$\Delta_T = 0$ is not cyclic on $\mathscr{N}(X)$.
On the other hand,  if $T \neq \lambda I_X$ for all $\lambda \in \mathbb C$,
then $\{I_{X^*}, T^*\}$ is a linearly independent set of 
$\ker(\Delta_{T^*}) = \ran \left(\Delta_T \right)^{\perp}$ in $ \mathscr{L}(X^*)$. 
Consequently the above condition implies that 
$\Delta_T$ can not be cyclic $\mathscr{N}(X) \to \mathscr{N}(X)$.
\end{proof}

Next we record a simple spectral condition which provides wide classes of non-hypercyclic commutator maps.

\begin{prop} \label{prop:nonEmptySpectra}
Let $X$ be a Banach space and $T \in \mathscr{L}(X)$.  If the point spectra of $T$ and $T^*$ are both nonempty then the commutator maps  $\Delta_T$ and 
$\Delta_{T}^*$ are not hypercyclic on $J$, respectively $J^*$, for any Banach ideal
$J \subset \mathscr{L}(X)$.
\end{prop}

\begin{proof}
Let $J \subset \mathscr{L}(X)$ be an arbitrary Banach ideal.
By assumption $Tx = \alpha x$ and $T^* x^* = \beta x^*$ for respective eigenvalues $\alpha, \beta \in \mathbb{C}$ corresponding to the normalized eigenvectors $x \in X$ and $x^* \in X^*$.  

We will first observe that the adjoint $\Delta^*_T \colon J^*\to J^*$ has a nonempty point spectrum.
In fact, define the continuous linear functional $\varphi \in J^*$ by $\varphi(U) = \langle x^*, Ux \rangle$, for $U \in J$. Note that $\varphi$ is bounded as
$\vert \varphi(U)\vert \le \Vert U\Vert_J$. 
For any $U \in J$ notice that
\begin{align*}
\langle \Delta^*_T(\varphi), \, U \rangle &= \langle \varphi,  \, TU - UT \rangle = 
\langle T^*x^*,  \, Ux \rangle - \langle x^*,  \, UTx \rangle \\
& = \beta \langle x^*,  \, Ux \rangle - \alpha \langle x^*,  \, Ux \rangle 
=  \left(\beta - \alpha\right) \langle \varphi,  \, U \rangle.
\end{align*}
Hence $\beta - \alpha$ is an eigenvalue of $\Delta_T^*$ on $J^*$ and $\Delta_T$ 
is not hypercyclic.

Moreover, apply $\left(\Delta_T \right)^{**} \colon J^{**} \to J^{**}$ to  
$x^* \otimes x \in J \subset J^{**}$, whence
\begin{equation*}
(\Delta_T)^{**}(x^* \otimes x) = x^* \otimes Tx - T^*x^* \otimes x 
= \left(\alpha - \beta \right) (x^* \otimes x).
\end{equation*}
In particular,  $\alpha - \beta$ is an eigenvalue of $(\Delta_T)^{**}$
 and hence $\Delta^*_T$ is not hypercyclic on $J^*$.
\end{proof}

To illustrate the above proposition consider a Hilbert space $H$ with orthonormal basis $(e_j)$ and
 the diagonal operator $D \in \mathscr{L}(H)$ defined as $De_j = \alpha_j e_j$, 
 where $(\alpha_j)$ is a bounded sequence of scalars and $j \geq 1$. Then 
the commutator map $\Delta_D$ is not hypercyclic on any separable Banach ideal $J \subset \mathscr{L}(H)$ and the same applies to $\Delta_D^*$ on the dual space $J^*$.

Our third observation implies that $\Delta_T$ is never hypercyclic
if $T$ is a  small operator, for instance if $T$ is compact. In fact,  we show more generally that  commutator maps  induced by Riesz operators are  never hypercyclic 
on any Banach ideal. This will be based on the result of Kitai \cite{Kit82} that every connected component of the spectrum of a hypercyclic operator intersects the unit circle 
(cf. \cite[Theorem 1.18]{BM09}). Recall that $T \in \mathscr{L}(X)$ is a \emph{Riesz} operator if its essential spectrum  $\sigma_{\mathrm{e}}(T) = \{ 0 \}$  
and that Riesz operators  are never hypercyclic~\cite[p.\ 160]{GEP11}.
The spectrum of a Riesz operator $T$ has the form
\begin{equation}  \label{eq:specRiesz}
\sigma(T) = \{ 0\} \cup \{ \lambda_n : n \geq 1 \}
\end{equation}
where $\{ \lambda_n : n \geq 1 \}$ is a discrete, at most countable, possibly empty set containing eigenvalues of finite multiplicity. 

The spectrum of the commutator map $\Delta_T$ on $\mathscr{L}(X)$ was computed 
by Lumer and Rosenblum~\cite{LR59}. In fact, the spectrum of $\Delta_T$ restricted to 
any Banach ideal $J \subset \mathscr{L}(X)$ satisfies the same formula
\begin{equation}  \label{eq:specCommI}
\sigma_J(\Delta_T) = \sigma(T) - \sigma(T) = \left\lbrace \lambda - \mu: \lambda, \mu \in \sigma(T) \right\rbrace
\end{equation}
for any $T \in \mathscr{L}(X)$,  where $\sigma_J(\Delta_T)$ denotes the spectrum of $\Delta_T$ 
considered as an operator $J  \to J$. For an argument of  \eqref{eq:specCommI} 
that applies to any Banach ideal $J$
see for instance the survey \cite[Theorem 3.12]{ST06}.

\begin{thm}  \label{thm:CommRieszNotHc}
Let $X$ be a Banach space and $J \subset \mathscr{L}(X)$ be a Banach ideal.  If 
$T \in \mathscr{L}(X)$ is a Riesz operator then the induced commutator map $\Delta_T$ is not hypercyclic on $J$.
\end{thm}

\begin{proof}
It follows from  \eqref{eq:specRiesz} and \eqref{eq:specCommI} that the spectrum of $\Delta_T$ on $J$ is given by
\begin{equation*}
\sigma_J (\Delta_T)  = \left\lbrace \lambda_n - \lambda_m : \lambda_n, \lambda_m \in \sigma(T), \; n, m \ge 0  \right\rbrace
\end{equation*}
where we define $\lambda_0 \coloneqq 0$. Here $\sigma_J(\Delta_T)$ is a closed and compact set 
which is at most countable. Since the map $(z,w) \mapsto z - w$ is continuous it follows that
 $\sigma_J(\Delta_T)$ is a discrete set containing the singleton $\{ 0 \}$ as a
connected component. Consequently the result of  Kitai~\cite{Kit82}  yields
that $\Delta_T$ is not hypercyclic on $J$.
\end{proof}

For the next corollary recall  that compact operators are examples of Riesz operators.

\begin{cor}\label{cor:comp}
Let $X$ be a Banach space and  $J \subset \mathscr{L}(X)$ be a separable Banach ideal.  
If $T \in \mathscr{L}(X)$ is a compact operator then the induced commutator map 
$\Delta_T$   is not hypercyclic on $J$.
\end{cor}

Note that if $T \in \mathscr{K}(X)$ then $\Delta_T = L_T - R_T$ is typically not a compact  
operator $\mathscr{K}(X) \to \mathscr{K}(X)$. 
In fact, if $X$ has the approximation property and $X$ is a reflexive Banach space,
then $\Delta_T = L_T - R_T$ is a 
weakly compact operator $\mathscr{K}(X) \to \mathscr{K}(X)$ whenever $T \in \mathscr{K}(X)$ (cf. \cite[Proposition 2.5]{ST06}).
We note that it is not difficult to check that $\Delta_T$ is not compact on $\mathscr{K}(\ell^2)$, where $T$ is the rank one operator $e_1 \otimes e_1$ on $\ell^2$.

To specify another standard class of examples contained in Theorem \ref{thm:CommRieszNotHc}
recall that the Banach space $X$ has the Dunford-Pettis property (DPP) if any weakly compact operator $S \colon X \to Y$, where $Y$ is an arbitrary Banach space, maps weak-null sequences of $X$ to norm-null sequences. In particular,
if $X$ has the DPP and $S \in \mathscr{L}(X)$ is weakly compact, then $S^2  \in \mathscr{K}(X)$
and $S$ is a Riesz operator. We refer to the survey \cite{Die80} for examples and properties of Banach spaces with the DPP. 

\begin{cor}\label{cor:DPP}
Let $X$ be a Banach space possessing the DPP and  $J \subset \mathscr{L}(X)$ be a separable Banach ideal.  
If $T \in \mathscr{L}(X)$ is a weakly compact operator then the commutator map 
$\Delta_T$  is not hypercyclic on $J$.
\end{cor}

\begin{rmk}
By using the related trace duality $\mathscr{K}(X)^* = \mathscr{N}(X^*)$ and 
arguing as in the proof of 
Proposition \ref{prop:nonEmptySpectra} one obtains the following fact: 
\emph{suppose that $X^*$ has the approximation property, $T \in \mathscr{N}(X^*)$ 
and $\{T, T^2\}$ is linearly independent. Then $\Delta_T$ is not a cyclic 
operator $\mathscr{K}(X) \to \mathscr{K}(X)$.}
\end{rmk}

\section{Commutator Maps of Hilbert Space Operators}  \label{sec:NonHcOpersHilbert}

In this section we provide examples of classes of operators on a separable infinite-dimensional 
Hilbert space $H$ over the complex field for which the corresponding commutator operator is not 
hypercyclic on suitable Banach ideals $J \subset \mathscr{L}(H)$.

Commutator operators induced by normal operators, or \emph{normal commutators},  were first studied by Anderson~\cite{And73b} and later  Maher~\cite{Mah92} and Kittaneh~\cite{Kit95}, among others, investigated them in the Banach ideal setting.
We consider normal commutator maps  acting on the ideal of Hilbert-Schmidt operators $C_2$. 
Recall that $C_2$ is complete in the Hilbert-Schmidt norm, which is defined as
\begin{equation*}
\lVert T \rVert^2_2 = \tr{T^*T} = \sum_k \langle T^* T e_k, e_k \rangle = \sum_k \langle Te_k, Te_k \rangle = \sum_k \lVert Te_k \rVert^2
\end{equation*}
where $T \in C_2$,  $\tr{T}$ denotes the trace of $T$ and $(e_n)$ is any orthonormal basis of  $H$.
The ideal  $C_2$ is a Hilbert space with the corresponding inner product 
\begin{equation*}
\langle S, T \rangle = \tr{T^* S}
\end{equation*}
for $S,T \in C_2$.
Further details on Hilbert-Schmidt operators can be found, for instance, in \cite{Con00}.

Let $X$ be a Banach space. Recall that  $U \in \mathscr{L}(X)$ is \emph{supercyclic} 
if there exists a vector $x \in X$ such that 
\begin{equation*}
\overline{ \left\lbrace \lambda U^n x : n \geq 0, \: \lambda \in \mathbb{C} \right\rbrace } = X.
\end{equation*}
The class of hypercyclic operators is strictly contained in the class of supercyclic operators~\cite[Example 1.15]{BM09}.   Kitai~\cite{Kit82} showed normal operators are 
never hypercyclic and it follows from a result of  Bourdon~\cite{Bou97} that
they are never even supercyclic (cf.  \cite[Theorem 5.30]{GEP11}). 

\begin{prop} \label{prop:commNrmlNeverHc}
Let $H$ be any Hilbert space and let $N \in  \mathscr{L}(H)$ be a normal operator.  
Then the induced commutator operator  $\Delta_N$ is not supercyclic on $C_2$.
\end{prop}

\begin{proof}
According to the above mentioned result of Bourdon it will be enough to verify that
the commutator operator $\Delta_N$ is a normal  $C_2 \to C_2$.  For this recall by the trace duality of Section \ref{sec:NonHcOpers} that $\Delta_{N^*} \colon C_2 \to C_2$
is the adjoint of the commutator map  $\Delta_N \colon C_2 \to C_2$.

Hence we get that 
\begin{align*}
\Delta_N\Delta_{N^*} &= (L_N - R_N)(L_{N^*} - R_{N^*}) \\
& = L_{NN^*} - L_N R_{N^*} - R_N L_{N^*} + R_{N^*N}  \\
& = L_{N^*N} -  L_{N^*} R_N  - R_{N^*} L_N  + R_{NN^*}  =  \Delta_{N^*}\Delta_N 
\end{align*}
and the result follows.
Note that above we used the facts that  left and right multipliers commute and the normality of $N$.
\end{proof}

The next corollary requires the \emph{comparison principle} which is a useful general 
tool formulated by Shapiro~\cite{Sha93}. 
Recall that the continuous map $T \colon X \to X$ is said to be a \emph{quasi-factor} of 
the continuous map  $T_0 \colon X_0 \to X_0$ if there exists a continuous map with dense range 
$\Psi \colon X_0 \to X$ such that $T\Psi = \Psi T_0$, that is the following diagram commutes.
\begin{equation*}
\xymatrix@C+2em@R+0.5em{ 					
X_0 \ar[r]^{T_0} \ar[d]_\Psi & X_0 \ar[d]^\Psi\\
X \ar[r]^T & X
}
\end{equation*}
When $T_0$ and $T$ are linear operators and the map $\Psi$ can be taken as linear, then we say $T$ is a \emph{linear quasi-factor} of $T_0$.
The comparison principle states that hypercyclicity is preserved by quasi-factors and 
supercyclicity is preserved by linear quasi-factors (cf. \cite[Section 1.1.1]{BM09}).

\begin{cor} \label{cor:commNrmlNotHcIdeals}
Let $J$ be a Banach ideal contained in $C_2$ and $N \in  \mathscr{L}(H)$  a normal operator.  Then the commutator  operator $\Delta_N$ is not supercyclic on $J$.
\end{cor}

\begin{proof}
The finite rank operators are contained in $J$ and they form a dense subset of $C_2$. 
Hence the inclusion map  $\Psi \colon J \rightarrow C_2$ is a linear map with dense range, for which
$\Delta_N \colon C_2 \to C_2$ is a quasi-factor of $\Delta_N \colon J \to J$ via the following commuting diagram.
\begin{equation*}
\xymatrix@C+2em@R+0.5em{ 
J \ar[r]^{\Delta_N} \ar[d]_\Psi & J \ar[d]^\Psi \\
C_2 \ar[r]^{\Delta_N} & C_2
}
\end{equation*}  
If $\Delta_N$ was supercyclic on $J$ then it would follow by the comparison principle that it would be supercyclic on $C_2$.  However we know from Proposition \ref{sec:NonHcOpers} that $\Delta_N$ is not supercyclic on $C_2$ and hence $\Delta_N$ cannot be supercyclic on $J$.
\end{proof}

\begin{rmk} \label{rmk:examplesOfNrmlComms}
To illustrate Corollary \ref{cor:commNrmlNotHcIdeals} we let $U \colon H \to H$ be  any unitary operator. The operator  $cU$ is  normal on $H$ for any constant $c \neq 0$ and hence the induced commutator map $\Delta_{U}$  is not supercyclic on any  Banach ideal $J \subseteq C_2$.  Concrete examples of such operators include the bilateral backward and forward shifts on $\ell^2(\mathbb{Z})$.  
\end{rmk}

We note that Kitai~\cite{Kit82} also showed hyponormal operators on the Hilbert space $H$ are never hypercyclic and Bourdon~\cite{Bou97} proved they are not even supercyclic. Recall that 
the operator $T \in \mathscr{L}(H)$ is \emph{positive}, denoted $T \geq 0$, if the inner product $\langle Tx, x \rangle \geq 0$ for all $x \in H$.  We say $T \in \mathscr{L}(H)$ is \emph{hyponormal} if $T^* T - T T^* \ge 0$
or equivalently if $\left\lVert Tx  \right\rVert \geq \left\lVert T^* x \right\rVert$ for all $x \in H$.
The class of hyponormal operators contains some well known classes of operators, for instance the subnormal, normal and self-adjoint operators~\cite{Hal82}. However, it follows from 
\cite[Exercise II.4.7]{Con91} that the commutator map 
$\Delta_T \colon C_2 \to C_2$ is hyponormal if and only if $T$ is normal. Thus hyponormality 
can not be directly employed to improve Proposition \ref{sec:NonHcOpers}.

A generalisation of hyponormality to the Banach space setting is as follows, the operator $T \in \mathscr{L}(X)$ is \emph{paranormal} if
\begin{equation*}
\Vert Ux \Vert^2  \leq  \Vert U^2x \Vert \cdot \Vert x \Vert
\end{equation*}
for all  $x \in X$.  It follows from Bourdon~\cite{Bou97} that paranormal operators are never supercyclic (see also \cite[p. 159]{GEP11}).

This suggests the following related question: if $T \in \mathscr{L}(\ell^2)$ is normal (or even hyponormal), does it follow 
that $\Delta_T$ is paranormal $\mathscr{K}(\ell^2) \to \mathscr{K}(\ell^2)$?
In the case when $T \in \mathscr{L}(\ell^2)$ is paranormal the following example illustrates that this does not necessarily hold.

\begin{eg}
We show there exists a paranormal operator $T \in \mathscr{L}(\ell^2)$ such that the induced commutator map $\Delta_T$ is not paranormal on the Banach ideals $C_2$ and $\mathscr{K}(\ell^2)$.

Let $T \in \mathscr{L}(\ell^2)$ be paranormal such that $T^*$ is not paranormal and $\ker(T) \neq \{0\}$.
We make the counter assumption that $\Delta_T$ is paranormal on $\mathscr{K}(\ell^2)$, which implies 
\begin{equation}  \label{eq:CommAnotPara}
\lVert TS - ST \rVert^2 \leq \lVert T^2S - 2TST + ST^2 \rVert  \lVert S\rVert
\end{equation}
for all $S \in \mathscr{K}(\ell^2)$.

Our choice of $T$ gives that there exists normalised $u, v \in \ell^2$ 
such that $Tv = 0$ and 
\begin{equation}  \label{eq:CommAnotPara2}
\lVert T^* u \rVert^2 > \lVert T^{*2} u \rVert.
\end{equation}
If we consider the rank one operator $S = u \otimes v$, observe that since $Tv = 0$ 
\begin{align*}
\lVert TS - ST \rVert^2 &= \lVert u \otimes Tv - T^* u \otimes v \rVert^2 = \lVert T^* u \rVert^2
\intertext{and}
\lVert T^2S - 2TST + ST^2 \rVert &= \lVert u \otimes T^2 v - 2 T^* u \otimes Tv + T^{*2}u \otimes v \rVert = \lVert T^{*2}u \rVert.
\end{align*}
By \eqref{eq:CommAnotPara} it follows that  that $\lVert T^* u \rVert^2 \leq \lVert T^{*2}u \rVert$, however this contradicts \eqref{eq:CommAnotPara2} and hence $\Delta_T$ is not paranormal on $\mathscr{K}(\ell^2)$. 
The same argument holds on $C_2$ since $\lVert u \otimes v \rVert_2 = \lVert u \rVert \lVert v \rVert$ and thus $\Delta_T$ is not paranormal on $C_2$.

Finally we show that there exist such  an operator $T \in \mathscr{L}(\ell^2)$.  We consider the forward shift $S \colon \ell^2 \to \ell^2$, which is paranormal while its adjoint the backward shift $S^* = B \colon \ell^2 \to \ell^2$ is not.  
The shift $S$ is an isometry on $\ell^2$ and by setting $Se_0 = 0$ we may extend it to $\ell^2 \oplus [e_0]$, which is isomorphic to $\ell^2$.  Thus we obtain a paranormal operator on $\ell^2 \oplus [e_0] \cong \ell^2$, with  a nontrivial  kernel  such that its adjoint is not paranormal.

\end{eg}

\section{Dynamics of the Commutator Map of the Backward Shift}  \label{sec:commBwNonHc}

In this section we study the dynamical properties of the commutator map $\Delta_{cB}$ on  
$\mathscr{K}(\ell^2)$ induced by scalar multiples $cB \colon \ell^2 \to \ell^2$ of the backward shift, 
where
\begin{equation*}
B(x_1,x_2,\dotsc) = (x_2, x_3,\dotsc)
\end{equation*}
for $(x_n) \in \ell^2$.
Recall from  Section \ref{sec:BgMultipliers} that the dynamical properties 
of the operator $T$ and  the multipliers  $L_T$ or $R_T$ acting on Banach ideals are related. In particular, one expects that
reasonable candidates for hypercyclic commutator maps $\Delta_T$
on $\mathscr{K}(\ell^2)$
should arise from operators $T$ which satisfy the Hypercyclicity Criterion
and induce commutator maps having (at least) dense range.  

The map $\Delta_B$ is a classical example of this kind: 
$cB$ satisfies the Hypercyclicity Criterion for $\lvert c \rvert > 1$, 
see \cite[Example 1.9]{BM09}, and 
$\overline{\Delta_B \left( \mathscr{K}(\ell^2) \right)} = \mathscr{K}(\ell^2)$.
In fact, in the trace duality of  \eqref{dual} 
one has that
\begin{equation*}
\left( \Delta_B \colon \mathscr{K}(\ell^2) \to \mathscr{K}(\ell^2) \right) ^* = 
\Delta_{S} \colon \mathscr{N}(\ell^2) \to \mathscr{N}(\ell^2)
\end{equation*}
where $S = B^* \in \mathscr{L}(\ell^2)$ is the forward shift 
$S(x_1,x_2,\dotsc) = (0,x_1,x_2,\dotsc)$ for $(x_n) \in \ell^2$. Moreover, 
it is a classical fact  that 
$U \notin \mathscr{K}(\ell^2)$ whenever 
$SU = US$ and $U \neq 0$, see \cite[Problem 147]{Hal82}, whence the annihilator
$\left\lbrace \Delta_B \left(\mathscr{K}(\ell^2) \right) \right\rbrace^\perp = \{ 0 \}$ in $\mathscr{N}(\ell^2)$. 

Our main result  demonstrates that $\Delta_{cB}$ is not hypercyclic on 
$\mathscr{K}(\ell^2)$ for any scalar $c$.  Observe that the point spectrum of $S = B^*$
is empty, so the general results from Section \ref{sec:NonHcOpers} do not apply here.
Our argument will instead explicitly demonstrate that  $\Delta_{cB}$ 
does not have a dense orbit in $\mathscr{K}(\ell^2)$.

We first prepare the setting for the argument.  Let $(e_n)$ be the orthonormal unit vector basis
in $\ell^2$ and  $A = (a_{i,j}) \in \mathscr{K}(\ell^2)$, where the matrix representation is with respect to $(e_n)$, that is, $a_{i,j} = \langle Ae_j,e_i \rangle = (A)_{i,j}$. Then the commutator map $\Delta_B(A)$ has the matrix representation
\begin{align}\label{matr}
\Delta_B(A) = BA - AB = 
\begin{pmatrix}
  a_{2,1} 	& a_{2,2} - a_{1,1} 	& a_{2,3} - a_{1,2} 		& \cdots  \\
    a_{3,1} & a_{3,2} - a_{2,1}	&  a_{3,3} -  a_{2,2}	& \cdots  \\
    a_{4,1} & a_{4,2} - a_{3,1} 	& a_{4,3} - a_{3,2} 		& \cdots  \\
  \vdots  	& \vdots 			& \vdots  				& \ddots  
\end{pmatrix} 
\end{align}
It is a known fact that $(e_m \otimes e_n)_{m, n \in \mathbb N}$ is a Schauder basis 
for $\mathscr{K}(\ell^2)$ in the shell-ordering $\max(m,n) = r$, where $r \in \mathbb N$ (see for instance \cite[Section 4.3]{Rya02}).
For each $k \ge 0$ let 
\begin{equation*}
D_k = [e_r \otimes e_{r+k}: r \ge 1]
\end{equation*}
where $[M]$ denotes the closed linear span in $\mathscr{K}(\ell^2)$ of the set 
$M \subset \mathscr{K}(\ell^2)$. In other words, $D_k$ consists of the compact operators
whose non-zero matrix elements are supported on the $k^\mathrm{th}$ (lower) subdiagonal.
It is well known that $D_k$ is complemented  in $\mathscr{K}(\ell^2)$ by the canonical
norm-$1$ projection $P_{D_{k}}$, for which 
\begin{equation*}
\left( P_{D_{k}}(B) \right)_{i,j}  = 
\begin{cases}
(B)_{r+k,r} & \textrm{if } (i,j) = (r+k,r) \textrm{ for some } r,\\
0 & \textrm{otherwise},
\end{cases}
\end{equation*}
see for instance \cite[Proposition 1.c.8]{LT77}.
Observe that  \eqref{matr} states that $\Delta_B$ maps the subdiagonal 
$D_k$ to the subdiagonal $D_{k-1}$ immediately above it for each $k \ge 1$.
This fact will be essential for the below argument.

\begin{thm} \label{thm:CommBwShiftNotHc}
Let $B \in \mathscr{L}(\ell^2)$ be the backward shift operator.  Then the commutator map
$\Delta_{cB}$ is not hypercyclic $\mathscr{K}(\ell^2) \to \mathscr{K}(\ell^2)$
for any constant $c$.
\end{thm}

\begin{proof}
Fix $c \neq 0$. We assume to the contrary 
that $A = (a_{i,j}) \in \mathscr{K}(\ell^2)$ is a hypercyclic vector 
for $\Delta_{cB}$.

We define a sequence  $f_k \colon \mathbb{D} \to \mathbb{C}$
of analytic functions on the open unit disk $\mathbb{D}$ by
\begin{equation*}
f_k(z) = \sum_{r=1}^\infty a_{k+r,r} z^{r-1}
\end{equation*}
where $z \in \mathbb{D}$,  $k \ge 0$  and the sequence $(a_{k+r,r})$ is composed of the matrix elements of 
the $k^\mathrm{th}$ subdiagonal of $A = (a_{i,j})$.  
Each function $f_k$ is clearly analytic in $\mathbb{D}$ since 
$(a_{k+r,r})_{r \geq 1}$ is a bounded sequence for $k \ge 0$.

Let $H(\mathbb{D})$ be the space of analytic functions on  $\mathbb{D}$.
 Next we define a transformation $\tau \colon H(\mathbb{D}) \to H(\mathbb{D})$ by
\begin{equation*}
(\tau g)(z) = b_1 + \sum_{r=2}^\infty \left(b_r - b_{r-1} \right) z^{r-1}
\end{equation*}
whenever $g(z) = \sum_{r=1}^\infty b_r z^{r-1} \in H(\mathbb{D})$.
Observe that if we apply $\tau$ to $f_k(z)$ we get
\begin{equation*}
(\tau f_k)(z) = a_{k+1,1} + \sum_{r=2}^\infty \left(a_{k+r,r} - a_{k+r-1,r-1} \right) z^{r-1},
\quad z \in \mathbb D.
\end{equation*}
In particular,  by  \eqref{matr} the map $\tau$ encodes the  manner in which 
$\Delta_B$ maps the matrix elements of $A$ on
the $k^\mathrm{th}$ subdiagonal $D_k$ to the subdiagonal $D_{k-1}$.
Furthermore,  notice for $z \in \mathbb{D}$ and $k \in \mathbb N$ that
\begin{align*}
(\tau f_k)(z) &=  a_{k+1,1} + (a_{k+2,2} - a_{k+1,1})z + (a_{k+3,3} - a_{k+2,2})z^2 + \dotsb  \\
&= (1-z)f_k(z).
\end{align*}
By $n$-fold iteration we get  that
\begin{equation}\label{iterate}
\left( \tau^n f_k \right)(z) = (\tau \circ \dotsb \circ \tau f_k)(z) =  (1-z)^n f_k(z)
\end{equation}
for all $z \in \mathbb{D}$,  $k \ge 0$ and $n \geq 1$.

Let $M_k = [e_m \otimes e_n:  1 \le m, n \le k]$ for $k \in \mathbb{N}$ and  
denote by  $P_k \colon \mathscr{K}(\ell^2) \to M_k$ the natural norm-$1$
projection defined  for $B = (b_{i,j}) \in \mathscr{K}(\ell^2)$ by $\left( P_k(B) \right)_{i,j} = b_{i,j}$ 
for $1 \le i,j \le k$ and $(P_k(B))_{i,j} = 0$ otherwise.
Next fix $\varepsilon > 0$ small enough so that 
\begin{equation}\label{var}
3\left|c \right| \varepsilon < 1.
\end{equation}
Since $A \in \mathscr{K}(\ell^2)$  it is well known that  
$\lim_k  \left\lVert A - P_k A \right\rVert = 0$, see for instance \cite[Theorem 4.4]{Con90}.
Hence there exists
 $k_\varepsilon \in \mathbb{N}$ such that $\left\lVert A - P_k A \right\rVert < \varepsilon$
 for all $k \geq k_\varepsilon$.  In particular, 
\begin{equation}\label{est}
\left|a_{k+r,r} \right| \leq  \left\lVert A - P_k A \right\rVert < \varepsilon
\end{equation} 
for $r \in \mathbb{N}$ and $k \ge k_{\varepsilon}$.

Recall by our counter assumption that  $A \in \mathscr{K}(\ell^2)$ is a hypercyclic 
vector for $\Delta_{cB}$.
Next we approximate the rank one
operator $e_1 \otimes e_1 = (d_{i,j}) \in \mathscr{K}(\ell^2)$, 
where $d_{1,1} = 1$ and $d_{i,j} = 0$ for  $i > 1$ or $j > 1$, that is 
\begin{equation*}
e_1 \otimes e_1 =
\begin{pmatrix}
   1 & 0 &  \cdots  \\
   0 & 0 &  \cdots  \\
  \vdots & \vdots &\ddots  
\end{pmatrix}
\end{equation*}
Hence it follows from our counter assumption that  there exists an integer $n > k_\varepsilon$ such that
\begin{equation*}
 \left\lVert \Delta_{cB}^n (A) - e_1 \otimes e_1 \right\rVert < \varepsilon.
\end{equation*}
Let $g_n \colon \mathbb D \to \mathbb C$ be the analytic function on  $\mathbb{D}$ 
derived from the above main diagonal $D_0$,
that is, $g_n(z) = \sum_{r=1}^\infty h_rz^{r-1}$, where
$h_r = \left( \Delta_{cB}^n (A) - e_1 \otimes e_1 \right)_{r,r}$ for $r \in \mathbb N$.
 Observe that 
\begin{equation*}
 \left\lVert P_{D_{0}} \circ (\Delta_{cB}^n (A) - e_1 \otimes e_1)_{|D_{k}} \right\rVert \le
  \left\lVert \Delta_{cB}^n (A) - e_1 \otimes e_1 \right\rVert < \varepsilon
\end{equation*}
so that $\vert h_r\vert < \varepsilon$ for each $r$.
Observe further that 
$g_n$ encodes the way $\Delta_{cB}^n (A) - e_1 \otimes e_1$ maps 
the subdiagonal $D_k$ to the main diagonal $D_0$, whence 
\begin{equation*}
g_n =  c^n \tau^n f_n  - 1
\end{equation*}
by  \eqref{iterate}. Moreover, it follows by a geometric series 
estimation  that
\begin{align} 
\left|g_n(z) \right| &\leq  \sum_{r=1}^\infty \left|h_r \right| \left|z \right|^{r-1}  
\leq \varepsilon \sum_{r=1}^\infty \left|z \right|^{r-1} 
\leq \frac{\varepsilon}{1 - \left|z \right|}, \quad z \in \mathbb D.	\label{eq:CommNotHcContradiction}
\end{align}
Recall further  by  \eqref{est} that $\left|a_{n+r,r} \right| < \varepsilon$ 
for all $r \in \mathbb{N}$ and $n > k_\varepsilon$.  
The geometric series estimate applied to  $f_n(z) = \sum_{r=1}^\infty a_{n+r,r} z^{r-1}$ 
yields for any $z \in \mathbb{D}$ satisfying 
$\left|z \right| \leq 1 - 3\varepsilon$ that one has 
\begin{equation} \label{eq:f_nUprBnd}
\left|f_n(z) \right| \leq \sum_{r=1}^\infty \left|a_{n+r,r} \right| \left|z \right|^{r-1} \leq \varepsilon \cdot \frac{1}{1 -(1 - 3\varepsilon)} = \frac{1}{3}.
\end{equation}
On the other hand, for $z_0 = 1 - 3\varepsilon$ we have by 
 \eqref{eq:CommNotHcContradiction} that
\begin{equation}  \label{eq:contradiction1}
\left|g_n(z_0) \right| {\leq}  \frac{\varepsilon}{1 - \left|z_0 \right|} 
= \frac{\varepsilon}{1 - (1 - 3\varepsilon)} = \frac{1}{3}.
\end{equation}
Recall next that $\lvert 3 c \varepsilon \rvert^n < 1$ by  \eqref{var} for 
$n \geq k_\varepsilon$ so that $\left|f_n(z_0) \right| \leq \frac{1}{3}$
by \eqref{eq:f_nUprBnd}, whence 
\begin{equation*}
\left|\left(3 c \varepsilon\right)^n f_n(z_0) \right| < \frac{1}{3}.
\end{equation*}
For the point $g_n(z_0)$ this means  that
\begin{equation} \label{eq:contradiction2}
\left|g_n(z_0) \right| = \left|c^n (1-z_0)^n f_n(z_0) - 1 \right| = \left|c^n \left(3\varepsilon\right)^n f_n(z_0) - 1 \right| > \frac{2}{3}.
\end{equation}

The fact that \eqref{eq:contradiction1}  and  \eqref{eq:contradiction2}  contradict each other 
implies that  $\Delta_{cB}$ can not be hypercyclic.
\end{proof}

\begin{rmk} 
 Let  $X = \ell^p$ for $1 < p < \infty$ or $X = c_0$.  The result in Theorem \ref{thm:CommBwShiftNotHc} also holds for the backward shift 
$B \in \mathscr{L}(X)$ and the commutator map $\Delta_{cB} \colon 
\mathscr{K}(X) \to \mathscr{K}(X)$. The argument for this extension is identical to the 
one for $X = \ell^2$. 
\end{rmk}

The simple observation that $\Delta_{I + cB} = \Delta_{cB}$ allows us to extend Theorem \ref{thm:CommBwShiftNotHc} to the family of operators 
$I + cB \colon \ell^2 \to \ell^2$, where $I$ is the identity operator on $\ell^2$ and $c \neq 0$.  This family of operators was shown to be hypercyclic by Salas~\cite{Sal95} and Le{\'o}n-Saavedra and Montes-Rodr{\'i}guez~\cite{LM97} proved they satisfy the Hypercyclicity Criterion.  
In fact more is true and below we see that
Theorem \ref{thm:CommBwShiftNotHc}  also holds for analytic polynomials 
in the backward shift $B$.  For any polynomial $p(z) = \sum_{j=0}^m c_j z^j$, where $m \in \mathbb{N}$ and $c_j \in \mathbb{C}$
for $0 \le j \le m$, we define the related operator $p(B) =  \sum_{j=0}^m c_j B^j \in 
\mathscr{L}(\ell^2)$. For example, if $p(z) = z + z^2$, then
\begin{equation*}
p(B)(x_1,x_2,\ldots ) = (x_2+x_3, x_3+x_4,\ldots), \quad x = (x_k) \in \ell^2.
\end{equation*}

\begin{thm} \label{thm:polyBwNonHc}
Let $p(B) \colon \ell^2 \to \ell^2$ be any analytic polynomial in the backward shift $B$.  
Then the induced commutator operator $\Delta_{p(B)}$ is not hypercyclic on 
$\mathscr{K}(\ell^2)$.
\end{thm}

\begin{proof}
Let $p(z) = \sum_{j=0}^m c_j z^j$, where $c_m \neq 0$  and $m \ge 1$. 
 The argument is a modification of the one for Theorem \ref{thm:CommBwShiftNotHc}.
 We will keep the notation of that argument 
and only indicate the required changes.

Observe that by  \eqref{matr} the commutator operator $\Delta_{B^j}$
maps the $k^\mathrm{th}$  subdiagonal  $D_k$ in $\mathscr{K}(\ell^2)$
to the $(k-j)^\mathrm{th}$ diagonal 
$D_{k-j}$  for each  $j = 1, \ldots, m$.
Note here that if $k - j < 0$ then $D_{k-j}$ is a superdiagonal in  $\mathscr{K}(\ell^2)$.

Assume again to the contrary 
that $A = (a_{i,j}) \in \mathscr{K}(\ell^2)$ is a hypercyclic vector 
for $\Delta_{p(B)}$. As above we define the analytic maps $f_k \in H(\mathbb{D})$ by
\begin{equation*}
f_k(z) = \sum_{r=1}^\infty a_{k+r,r} z^{r-1}, \quad z \in \mathbb D
\end{equation*}
which are related to the $k^\mathrm{th}$  subdiagonal  of $A = (a_{i,j})$ for $k \ge 0$. 
For each fixed $j  \ge 2$ 
define the corresponding transformation $\tau_j \colon H(\mathbb{D}) \to H(\mathbb{D})$,
which acts on  $f_k$ by
\begin{align*}
(\tau_j f_k)(z) & = a_{k+1, 1} + \dotsb + a_{k+j,j} z^{j-1} + 
\sum_{r=j+1}^\infty \left( a_{k+r,r} - a_{k+r-j, r-j} \right) z^{r-1}\\
& = (1 - z^j)f_k(z).
\end{align*}
In this case   $n$-fold iteration satisfies
\begin{equation}\label{iterate2}
(\tau_j^n f_k)(z) = (1-z^j)^n f_k(z), \quad z \in \mathbb D
\end{equation}
for each $j$  and $k$ satisfying $1 \le j \le m$ and $k \ge n \cdot j$.

Let $\varepsilon > 0$ be given. Since $A \in \mathscr{K}(\ell^2)$ is a hypercyclic vector 
for $\Delta_{p(B)}$ we may  choose (following  the proof of 
Theorem \ref{thm:CommBwShiftNotHc} and its notation) a power $n \ge k_\varepsilon$ so that
\begin{equation*}
 \left\lVert \Delta_{p(B)}^n (A) - e_1 \otimes e_1 \right\rVert < \varepsilon.
\end{equation*}
 Let $g_n(z) = \sum_{r=1}^\infty h_rz^{r-1}$ be the analytic map derived from 
 the main diagonal of the operator $\Delta_{p(B)}^n (A) - e_1 \otimes e_1$.
The crucial observation is that the analytic map 
\begin{equation*}
g_n = (c_m)^m (1-z^m)^n f_{mn} - 1 
\end{equation*}
on $\mathbb{D}$ precisely encodes the manner in which $\Delta_{p(B)}^n (A) - e_1 \otimes e_1$ 
maps the subdiagonal  $D_{mn}$ to $D_0$. Here $\vert h_r\vert < \varepsilon$
for each $r \in \mathbb N$, since  again
\begin{equation*}
 \left\lVert P_{D_{0}} \circ (\Delta_{p(B)}^n (A) - e_1 \otimes e_1)_{| D_{mn}} \right\rVert
 \le  \left\lVert \Delta_{p(B)}^n (A) - e_1 \otimes e_1 \right\rVert < \varepsilon.
\end{equation*}
Following these preparations one finds the desired contradiction 
once $\varepsilon > 0$ is small enough
as in the proof of  Theorem \ref{thm:CommBwShiftNotHc}
by using \eqref{iterate2}.
\end{proof}

The following corollary shows that Theorems \ref{thm:CommBwShiftNotHc} and  \ref{thm:polyBwNonHc} 
also hold on any Banach ideals contained in $\mathscr{K}(\ell^2)$.

\begin{cor}
Let $p(B) \colon \ell^2 \to \ell^2$ be any analytic polynomial in the backward shift $B$.    
Then the commutator map $\Delta_{p(B)}$ is not hypercyclic on any 
Banach ideal $J \subset \mathscr{K}(\ell^2)$.
\end{cor}

\begin{proof}
The finite rank operators are contained in $J$ and they form a dense subset of $\mathscr{K}(\ell^2)$. Hence the inclusion map $\Psi \colon J \rightarrow \mathscr{K}(\ell^2)$ 
is continuous with dense range, so that $\Delta_{p(B)}$ on $\mathscr{K}(\ell^2)$
is a quasi-factor of $\Delta_{p(B)}$ on $J$
 via the following commuting diagram.
\begin{equation*}
\xymatrix@C+3em@R+0.8em{ 
J \ar[r]^{\Delta_{p(B)}} \ar[d]_\Psi & J \ar[d]^\Psi\\
\mathscr{K}(\ell^2) \ar[r]^{\Delta_{p(B)}} & \mathscr{K}(\ell^2)
}
\end{equation*}
By arguing as in Corollary \ref{cor:commNrmlNotHcIdeals} it follows that $\Delta_{p(B)}$ cannot be hypercyclic on $J$.
\end{proof}

Following from Theorem \ref{thm:CommBwShiftNotHc}, it is natural to consider the case of the bilateral backward shift $B \colon \ell^2(\mathbb{Z}) \to \ell^2(\mathbb{Z})$.  However we know from Remark \ref{rmk:examplesOfNrmlComms} that the induced commutator map $\Delta_{B}$ cannot be  supercyclic on any Banach ideal $J \subseteq C_2$, since $B$ is unitary on $\ell^2(\mathbb{Z})$.  
So in particular  $\Delta_{cB}$ is not hypercyclic on $J$ for any $c \neq 0$.
We note however that the results of Sections \ref{sec:NonHcOpersHilbert} and \ref{sec:commBwNonHc} do not apply to the case $\Delta_{cB} \colon \mathscr{K}\left( \ell^2(\mathbb{Z}) \right) \to \mathscr{K}\left( \ell^2(\mathbb{Z}) \right)$.

Another class of reasonable candidates to give hypercyclic commutator maps are the dual hypercyclic operators.
We recall that $T \in \mathscr{L}(X)$ is  \emph{dual hypercyclic} if both $T$ and its adjoint $T^*$ are hypercyclic and the first examples were obtained	 by Salas~\cite{Sal91,Sal07} and  Petersson~\cite{Pet06}.  

This approach is further supported by the correspondence, established by Bonet et al.~\cite{BMP04}, between the hypercyclic properties of the multipliers $L_T$, $R_T$ and, respectively, the operators $T$ and $T^*$.  However we will see dual hypercyclic operators do not necessarily induce hypercyclic commutator maps.

\begin{eg}
We first observe that there exists a dual hypercyclic operator such that the induced commutator map is not hypercyclic on any separable Banach ideal $J \subseteq \mathscr{K}\left(\ell^2(\mathbb{Z}) \right)$.

We consider the dual hypercyclic operator from \cite{Sal07} which  is of the form $I + T \colon \ell^2(\mathbb{Z}) \to \ell^2(\mathbb{Z})$, where $T$ is a weighted backward shift and $I$ is the identity operator.  Note from \cite[Remark 3.2]{Sal07} that the spectrum $\sigma\left(I + T \right) = \{ 1 \}$. 
If we consider the induced commutator map $\Delta_{I+T}$ on a separable Banach ideal $J \subseteq \mathscr{K}\left( \ell^2(\mathbb{Z}) \right)$, then it follow from   \eqref{eq:specCommI} that its spectrum is $\sigma_J\left( \Delta_{I+T} \right) = \{0 \}$. Since it does not intersect the unit circle it follows by Kitai's condition, (cf. \cite[Theorem 1.18]{BM09}), that it is not hypercyclic.

Next we recall that  the dual hypercyclic operators constructed in \cite{Sal91} and \cite{Pet06}  are invertible weighted backward shifts on, respectively, $\ell^2(\mathbb{Z})$ and $\ell^p(\mathbb{N})$, for $1<p<\infty$.  It is  known from the survey of Shields~\cite[Section 5]{Shi74} that their spectra are annuli centred at the origin and hence the above argument does not apply to these particular operators.
\end{eg}

We conclude by mentioning some natural questions that arise from our study.
If we consider the question of supercyclicity of $\Delta_{cB}$, the argument of Theorem \ref{thm:CommBwShiftNotHc} does not seem applicable.
So one question arising from Theorem \ref{thm:CommBwShiftNotHc} is whether $\Delta_{B} \colon \mathscr{K}(\ell^2) \to \mathscr{K}(\ell^2)$ is  supercyclic?

The main remaining question  is does there exist a separable Banach ideal $J \subset \mathscr{L}(X)$ and 
$T \in \mathscr{L}(X)$ such that the commutator map $\Delta_T \colon J \to J$ is hypercyclic?
Are the dual hypercyclic operators from \cite{Sal91} and \cite{Pet06} good candidates to induce such a hypercyclic commutator map?
It is also possible to consider weaker topologies on $\mathscr{L}(\ell^2)$ and ask whether there exists $T \in \mathscr{L}(\ell^2)$ such that
$\Delta_T$ is hypercyclic on $\mathscr{L}(\ell^2)$ with respect to the strong operator topology.

\bibliographystyle{plain}

\bibliography{/Users/gilmore/Dropbox/Latex/bibliography/elOps,/Users/gilmore/Dropbox/Latex/bibliography/hypercyclic,/Users/gilmore/Dropbox/Latex/bibliography/FAgeneral}

\end{document}